\documentclass[12pt,reqno]{amsart}
\usepackage{amsmath}
\usepackage{amsfonts}
\usepackage{amssymb}
\usepackage[all,cmtip]{xy}           

\usepackage{bbm}
\usepackage{bbding}
\usepackage{txfonts}
\usepackage[shortlabels]{enumitem}
\usepackage{ifpdf}
\ifpdf
\usepackage[colorlinks,final,backref=page,hyperindex]{hyperref}
\else
\usepackage[colorlinks,final,backref=page,hyperindex,hypertex]{hyperref}
\fi
\usepackage{tikz-cd}
\usepackage[active]{srcltx}
\usepackage{tikz}
\usepackage{amscd}
\usepackage{bm}
\usepackage{mathrsfs}

\makeatletter

\newtheorem{thm}{Theorem}[section]
\newtheorem{lem}[thm]{Lemma}
\newtheorem{cor}[thm]{Corollary}
\newtheorem{prop-def}[thm]{Proposition-Definition}
\newtheorem{conj}[thm]{Conjecture}
\theoremstyle{definition}
\newtheorem{defn}[thm]{Definition}

\newtheorem{remark}[thm]{Remark}

\newcommand{\nc}{\newcommand}

\nc{\delete}[1]{{}}

\nc{\mlabel}[1]{\label{#1}}  
\nc{\mcite}[1]{\cite{#1}}  
\nc{\mref}[1]{\ref{#1}}  
\nc{\meqref}[1]{\eqref{#1}}  
\nc{\mbibitem}[1]{\bibitem{#1}} 

\delete{
	\nc{\mlabel}[1]{\label{#1} {{\emph{{\ }\ (#1)}}}}				 
	\nc{\mcite}[1]{\cite{#1}{{\emph{{\ }(#1)}}}}  
	\nc{\mref}[1]{\ref{#1}{{\emph{{\ }(#1)}}}}  
	\nc{\meqref}[1]{\eqref{#1}{{\emph{{\ }(#1)}}}}  
	\nc{\mbibitem}[1]{\bibitem[\bf #1]{#1}} 
}

\nc{\mrm}[1]{{\rm #1}}
\nc{\name}[1]{{\bf #1}}
\nc{\tforall}{\ \text{for all }}
\nc{\gldim}{\mathrm{gldim}}
\nc{\la}{\longrightarrow}
\nc{\ot}{\otimes}
\nc{\rar}{\rightarrow}

\newcommand{\Mod}{\mathrm{Mod~}}
\newcommand{\proj}{\mathrm{proj~}}

\nc{\Alg}{{\mathrm{Alg}}}
\nc{\bfk}{{\bf k}}
\nc{\C}{{\mathrm{C}}}
\nc{\DA}{{\mathsf{DA}_\lambda}}
\nc{\Dif}{{{}_\lambda\!\mathfrak{Dif}}}
\nc{\Difinfty}{{{}_\lambda\!\mathfrak{Dif}_\infty}}
\nc{\DO}{{\mathsf{DO}_\lambda}}
\nc{\End}{\mrm{End}}
\nc{\Ext}{\mrm{Ext}}
\nc{\Fil}{\mrm{Fil}}
\nc{\Fr}{\mrm{Fr}}
\nc{\Frob}{\mrm{Frob}}
\nc{\Gal}{\mrm{Gal}}
\nc{\GL}{\mrm{GL}}

\nc{\Hom}{\mrm{Hom}}
\nc{\Tor}{\mrm{Tor}}
\nc{\Hoch}{\mrm{Hoch}}
\nc{\HC}{\mrm{HC}}
\nc{\hsr}{\mrm{H}}
\nc{\hpol}{\mrm{HP}}
\nc{\im}{\mrm{Im}}
\nc{\Irr}{\mrm{Irr}}
\nc{\incl}{\mrm{incl}}
\nc{\length}{\mrm{length}}
\nc{\NLSW}{\mrm{NLSW}}
\nc{\Lie}{\mrm{Lie}}
\nc{\mchar}{\rm char}
\nc{\mpart}{\mrm{part}}
\nc{\ql}{{\QQ_\ell}}
\nc{\qp}{{\QQ_p}}
\nc{\rank}{\mrm{rank}}
\nc{\rcot}{\mrm{cot}}
\nc{\rdef}{\mrm{def}}
\nc{\rdiv}{{\rm div}}
\nc{\rmH}{ {\mathrm{H}}}
\nc{\rtf}{{\rm tf}}
\nc{\rtor}{{\rm tor}}
\nc{\res}{\mrm{res}}
\nc{\Sh}{{\mathrm{Sh}}}
\nc{\SL}{\mrm{SL}}
\nc{\Spec}{\mrm{Spec}}
\nc{\sgn}{{\mathrm{sgn}}}
\nc{\tor}{\mrm{tor}}
\nc{\Tr}{\mrm{Tr}}
\nc{\tr}{\mrm{tr}}
\nc{\wt}{\mrm{wt}}
\nc{\op}{\mrm{op}}

\nc{\cpx}[1]{#1^{\bullet}}

\nc{\HH}{ \mathrm{HH}} \nc{\TP}{\widetilde{P}}
\nc{\bbA}{{\mathbb A}}   \nc{\bbB}{{\mathbb B}}
\nc{\bbC}{{\mathbb C}}
\nc{\bbD}{{\mathbb D}}   \nc{\bbE}{{\mathbb E}}
\nc{\bbF}{{\mathbb F}}   \nc{\bbG}{{\mathbb G}}
  \nc{\bbL}{{\mathbb L}}
\nc{\bbN}{{\mathbb N}}   \nc{\bbP}{{\mathbb P}}
\nc{\bbQ}{{\mathbb Q}}   \nc{\bbR}{{\mathbb R}}
\nc{\bbT}{{\mathbb T}}   \nc{\bbV}{{\mathbb V}}
\nc{\bbZ}{{\mathbb Z}}
\nc{\bbm}{{\mathbb m}}   \nc{\bbS}{{\mathbb S}}

\nc{\calA}{{\mathcal A}}    \nc{\calc}{{\mathcal C}}
\nc{\calD}{\mathcal{D}}     \nc{\cale}{{\mathcal E}}
\nc{\calf}{{\mathcal F}}    \nc{\calg}{{\mathcal G}}
\nc{\calH}{{\mathcal H}}    \nc{\cali}{{\mathcal I}}
\nc{\call}{{\mathcal L}}    \nc{\calm}{{\mathcal M}}
\nc{\caln}{{\mathcal N}}    \nc{\calo}{{\mathcal O}}
\nc{\calP}{{\mathcal P}}    \nc{\calr}{{\mathcal R}}
\nc{\cals}{{\mathcal S}}    \nc{\calT}{{\mathcal T}}
\nc{\calv}{{\mathcal V}}    \nc{\calw}{{\mathcal W}}
\nc{\calx}{{\mathcal X}}

\newcommand{\scrD}{\mathscr{D}}
\newcommand{\scrK}{\mathscr{K}}

\nc{\fraka}{{\mathfrak a}}
\nc{\frakb}{\mathfrak{b}}
\nc{\frakg}{{\frak g}}
\nc{\frakl}{{\frak l}}
\nc{\fraks}{{\frak s}}
\nc{\frakB}{{\frak B}}
\nc{\frakm}{{\frak m}}
\nc{\frakM}{{\frak M}}
\nc{\frakp}{{\frak p}}
\nc{\frakW}{{\frak W}}
\nc{\frakX}{{\frak X}}
\nc{\frakS}{{\frak S}}
\nc{\frakA}{{\frak A}}
\nc{\frakC}{{\frak{C}}}
\nc{\frakx}{{\frakx}}
\nc{\frakt}{{\mathfrak{T}}}

\nc{\Tria}{\mathrm{Tria}}

\nc{\vspa}{\vspace{-.1cm}}
\nc{\vspb}{\vspace{-.2cm}}
\nc{\vspc}{\vspace{-.3cm}}
\nc{\vspd}{\vspace{-.4cm}}
\nc{\vspe}{\vspace{-.5cm}}

\nc{\lir}[1]{\textcolor{red}{\underline{Li:}#1 }}

\begin{document}

\title[Han's conjecture]{A recollement approach to Han's conjecture}

\author{Ren Wang, Xiaoxiao Xu, Jinbi Zhang,  and Guodong Zhou}

\address{Ren Wang, School of Mathematics, Hefei University of Technology, Hefei 230000, China}
\email{renw@mail.ustc.edu.cn}

\address{Xiaoxiao Xu and Guodong Zhou, School of Mathematical Sciences, Ministry of Education Key Laboratory of Mathematics and Engineering Applications, Shanghai Key Laboratory of PMMP,  East China Normal University, Shanghai 200241, China}
\email{52275500015@stu.ecnu.edu.cn, gdzhou@math.ecnu.edu.cn}

\address{Jinbi Zhang, School of Mathematical Sciences, Anhui University, Hefei 230601, China}
\email{zhangjb@ahu.edu.cn}

\date{\today}

\begin{abstract}
A conjecture due to  Y. Han  asks whether that Hochschild homology groups of a finite dimensional algebra vanish for sufficiently large degrees would imply that the algebra is of finite global dimension. We investigate this conjecture from the  viewpoint of recollements  of derived categories. It is shown  that for a recollement of unbounded derived categories of  rings which extends downwards (or upwards) one step, Han's conjecture holds for the ring in the middle if and only if  it holds for the two rings on the two sides and hence
 Han's conjecture is reduced  to derived $2$-simple rings.
Furthermore, this reduction result is applied to  Han's conjecture for  Morita contexts rings and exact contexts. Finally it is proved  that Han's conjecture holds for skew-gentle algebras, category algebras of finite EI categories and Geiss-Leclerc-Schr\"{o}er algebras associated to Cartan triples.

\end{abstract}

\subjclass[2020]{
16E40 
18G80 
19D55 
16E10 
16E35 
16E45 
16E65 
 }
\keywords{Exact context, Finite EI category, Global dimension, GLS algebra, Han's conjecture, Hochschild homology,   Morita context  ring,  Recollement,   Skew-gentle algebra}

\maketitle

\vspace{-.7cm}

 \tableofcontents

\vspace{-.7cm}

\allowdisplaybreaks

\section{Introduction}
Let $A$ be a finite dimensional algebra over a field. It  is well known that if the global dimension of $A$ is finite, then its Hochschild cohomology groups $\HH^n(A)$ vanish for all sufficiently large $n$. `` Happel's question " (see \cite{Hap89}) is concerned with the converse of this statement, that is, if the Hochschild cohomology groups  $\HH^n(A)$ vanish  for all sufficiently large $n$,  then is its global dimension finite? This does not hold in general by an example in \cite{BGMS05} due to R.-O Buchweitz, E. L. Green, D. Madsen, and {\O}. Solberg. 
In \cite{Han06}, Y. Han showed that all the higher Hochschild homology of this counter-example do not vanish. This led him to suggest the following conjecture:
\begin{conj}[Han's conjecture]
  Let $A$ be a finite dimensional algebra over an algebraically closed field. If the Hochschild homology groups $\HH_n(A)$ vanish for all sufficiently large $n$, then its global dimension is finite.\end{conj}

Y. Han showed in the same paper \cite{Han06} that the conjecture is true for monomial algebras. In \cite{ALVP92}, L. L. Avramov and M. Vigu\'e-Poirrier proved that finitely generated commutative algebras satisfy Han's conjecture. Since then, it has been  shown that Han's conjecture holds for many classes of algebras; see  \cite{BE08,  BM09, BM10, SVP10,   BHM12, SSAV13, IS14,  BM17, IM22}.

 Recently, C. Cibils, M. Lanzilotta, E. N. Marcos, M. Redondo,  and A. Solotar  initiated a project  to study Han's conjecture,   firstly via  the operations of deleting or adding arrows in bounded quivers (see \cite{CLMS20a}),   then  consider
the behavior of Han's conjecture under  split, left or right,  bounded  extensions of algebras (see \cite{CLMS20b, CRS21,  CLMS22}), and finally considered  Morita context algebras (see \cite{CLMS24}). Their main tool to control Hochschild homology is  the Jacobi-Zariski long nearly exact sequence (see \cite{CLMS21}).

 We refer the reader to \cite{Cru23} and the references therein for some new advances on the conjecture.

 Note that Han's conjecture could be stated for an arbitrary ring although this conjecture is not true in this generality; for counter-examples see \cite{Cru23}.



In this paper we     investigate Han's conjecture from the viewpoint of recollement of derived categories. 

Following the idea to reduce homological conjectures via recollements in the sense of \cite{BBD81} or ladders in the sense of \cite{BGS88} of derived categories (see \cite{QH16, CX17}),  our first result is a reduction theorem for Han's conjecture.
\begin{thm}\label{thm:reduction}
Let $R$, $S$ and $T$ be rings which  admit  a ladder of height $2$:
$$\xymatrixcolsep{4pc}
\xymatrix{
\scrD(S) \ar@<0.8ex>[r]\ar@<-2.4ex>[r]
&\scrD(R) \ar@<-2.4ex>[r]  \ar@<-2.4ex>[l]
\ar@<0.8ex>[l] \ar@<0.8ex>[r]
&\scrD(T), \ar@<-2.4ex>[l] \ar@<0.8ex>[l]
}$$
where $\scrD(R)$ is the unbounded derived category of all right $R$-modules.
Then Han's conjecture holds for $R$ if and only if it holds for $S\times T$. So Han's conjecture is reduced to derived $2$-simple rings.
\end{thm}
Here a  ring $R$ is called \emph{derived 2-simple} if it does not admit any ladder of height 2 by  derived categories of rings.

The easy proof of the above reduction theorem has two ingredients. The first ingredient is \cite[Theorem I]{AKLY17a} which relates global dimensions of rings in a recollement. The second ingredient is a splitting theorem for Hochschild homology due to B. Keller.
\begin{thm}[Keller \cite{Kel98}]\label{thm:split}
Let $R$, $S$ and $T$ be rings with a ladder of height $2$:
$$\xymatrixcolsep{4pc}
\xymatrix{
\scrD(S) \ar@<0.8ex>[r]\ar@<-2.4ex>[r]
&\scrD(R) \ar@<-2.4ex>[r]  \ar@<-2.4ex>[l]
\ar@<0.8ex>[l] \ar@<0.8ex>[r]
&\scrD(T). \ar@<-2.4ex>[l] \ar@<0.8ex>[l]
}$$
Then
$$\HH_n(R)\simeq \HH_n(S) \oplus  \HH_n(T)
\  \mbox{for all }\  n\in \mathbb{N}.$$
\end{thm}
We state Theorem \ref{thm:split} in the more general setting of homology theories of dg rings; see Theorem \ref{thm:splitting-dg}.

\medskip

Another reduction method for Han's conjecture via algebra extensions are presented in recent papers \cite{CLMS20a, CLMS20b, CRS21, CLMS22, CLMS24, IM22,  QXZZ24}.

In Section \ref{subsec-han-conj-Morita}, we apply Theorem \ref{thm:reduction} to   Morita context rings.

\begin{thm}[Theorem \ref{Hanconj-Morita context}]
\label{Intro-Hanconj-Morita context}
Let $R:=\left(\begin{smallmatrix}                                            S & _SN_T \\                                              _TM_S & T                                            \end{smallmatrix}\right)_{(\alpha, \beta)}$ be a Morita context ring.
\begin{itemize}
\item[$(1)$] If $\beta:M\otimes_SN\to T$ is a monomorphism, ${\rm Tor}^S_i(M,N)=0$ for all $i\ge 1$, and $M_S\in \scrD^c(S)$, then Han's conjecture holds for $R$ if and only if it holds for $S\times T/{\rm Im}(\beta)$.

\item[$(2)$] If $\alpha:N\otimes_TM\to S$ is a monomorphism, ${\rm Tor}^T_i(N,M)=0$ for all $i\ge 1$, and $N_T\in \scrD^c(T)$, then Han's conjecture holds for $R$ if and only if it holds for $S/{\rm Im}(\alpha)\times T$.

    \end{itemize}
\end{thm}
Examples of Morita context rings satisfying the conditions of the result above are constructed.

\medskip

In section \ref{subsec-han-conj-exactcontext}, we apply Theorem \ref{thm:reduction} to   exact contexts.

Let $\lambda:R\to S$ and $\mu:R\to T$ be two ring homomorphisms. Assume that $M$ is an $S$-$T$-bimodule together with a fixed element $m\in M$. The quadruple $(\lambda, \mu, M, m)$ is said to be an \emph{exact context} if the following sequence
$$0\longrightarrow R \stackrel{\left({\lambda}\atop{\mu}\right)}{\longrightarrow}
S\oplus T
\stackrel{(\cdot m,-m \cdot)}{\longrightarrow}
M\longrightarrow 0
$$
is an exact sequence of abelian groups, where $\cdot m$ and $m \cdot$ denote the right and left multiplication by $m$, respectively.
From an exact context $(\lambda,\mu, M, m)$, H. X. Chen and C. C. Xi  associated it with a new ring $T\boxtimes_RS$, called the \emph{noncommutative tensor product} of $(\lambda,\mu, M, m)$ (see \cite{CX19, CX21}), which  generalizes the usual tensor products over commutative rings, and captures coproducts of rings and dual extensions.
Recall that an exact context $(\lambda,\mu, M, m)$ is said to be \emph{homological} if ${\rm Tor}^R_i(T,S)=0$ for all $i\geq 1$.

\begin{thm}[Theorem \ref{Hanconj-exact context}]
\label{Intro-Hanconj-exact context}
Let $(\lambda, \mu, M, m)$ be a homological exact context. Assume $T_R\in\scrD^c(R)$.
Then Han's conjecture holds for $(T\boxtimes_RS)\times R$ if and only if it holds for $S\times T$.
\end{thm}

In Section~\ref{sec-skew-gentle}, we apply Theorem \ref{thm:reduction}  to skew-gentle algebras and prove that Han's conjecture holds for this class of algebras.

In Section~\ref{sec-category-algebra}, we apply Theorem \ref{thm:reduction}  to category algebras of finite EI categories and GLS algebras associated with Cartan matrices, and show that Han's conjecture holds for these two classes of algebras.
\medskip

%

Throughout this paper, let $\bfk$ be a commutative ring. All rings are $\bfk$-algebras and ring homomorphisms are $\bfk$-linear. All triangulated categories and triangle functors are also assumed to be $\bfk$-linear.

\bigskip

\section{Preliminaries}
In this section, we shall   recall notations, definitions and  basic facts about recollements and ladders of triangulated categories, and  Keller's localisation theorem and splitting theorem for  Hochschild homology.

\subsection{Recollements and ladders}\

In this subsection, we recall and establish some basic results on recollements of triangulated
categories. This notion was  introduced by  A. A. Beilinson, J. Bernstein, and P. Deligne in \cite{BBD81} to describe intersection homology and  perverse sheaves over singular spaces.

\begin{defn}(\cite[1.4]{BBD81})
Let $\mathcal{T}$, $\mathcal{T}'$ and $\mathcal{T}''$ be triangulated $k$-categories.
  A \emph{recollement} of $\mathcal{T}$ by $\mathcal{T'}$ and $\mathcal{T''}$ is a diagram of six triangle functors
  \begin{align}\label{diag:rec}
   \xymatrixcolsep{4pc}\xymatrix{\mathcal{T'} \ar@<0ex>[r]|{i_*=i_!}  &\mathcal{T} \ar@<-2ex>[l]|{i^*}
\ar@<2ex> [l]|{i^!}  \ar@<0ex>[r]|{j^!=j^*}  &\mathcal{T''} \ar@<-2ex>[l]|{j_!} \ar@<2ex>[l]|{j_{*}}
  }
  \end{align}
which satisfies the following conditions:
\begin{itemize}
\item[(R1)] $(i^\ast,i_\ast)$,\,$(i_!,i^!)$,\,$(j_!,j^!)$ ,\,$(j^\ast,j_\ast)$
are adjoint pairs;
\item[(R2)] $i_*,~j_*,~j_!$ are fully faithful$;$
\item[(R3)] $j^*\circ i_*=0$ (thus $i^*\circ j_!=0$ and $i^!\circ j_*=0$);
\item[(R4)] for any object $X$ in $\mathcal{T}$, there are two triangles in $\mathcal{T}$ induced by counit and unit adjunctions$:$
    $$ \xymatrix@R=0.5pc{i_*i^!X \ar[r] &X \ar[r] & j_*j^*X \ar[r]& i_*i^!X[1],\\
    j_!j^*X \ar[r] &X \ar[r] & i_*i^*X \ar[r]& j_!j^*X[1].} $$
\end{itemize} \end{defn}

Let us recall the language of compactly generated triangulated categories.
Assume that  a triangulated category $\mathcal{T}$ admits small coproducts (that is, coproducts indexed over sets exist  in $\mathcal{T}$). An object $X$ of $\mathcal{T}$ is called \emph{compact} if $\Hom_{\mathcal{T}}(X,?)$ commutes with infinite direct sums. The full subcategory of $\mathcal{T}$ consisting of all compact objects is denoted by $\mathcal{T}^c$. For a class  $\mathcal{S}$ of objects of $\mathcal{T}$, We write ${\rm Tria}(\mathcal{S})$ for the smallest triangulated subcategory of $\mathcal{T}$ containing $\mathcal{S}$ and closed under taking direct sums.
A set $\mathcal{S}$ is said to be a set of \emph{compact generators} of $\mathcal{T}$ if
all objects in $\mathcal{S}$ are compact and $\mathcal{T}={\rm Tria}(\mathcal{S})$. In this case, $\mathcal{T}$ is called a \emph{compactly generated triangulated category}.
Examples of compactly generated triangulated categories include unbounded derived categories $\scrD(R):=\scrD(\Mod R)$ of the category $\Mod R$ of all right  modules over a (differential graded) ring $R$ whose subcategory of compact objects is the  bounded homotopy category of  finitely generated projective modules $\scrK^b(\proj R)$ and unbounded derived categories of quasi-coherent sheaves $\scrD(\mathrm{Qch}(X))$ on a nice scheme $X$  whose subcategory of compact objects are the full subcategory of  perfect complexes $\mathrm{per}(X)$.

Recall that a sequence of triangulated categories $\mathcal{T'}\stackrel{F}{\to}\mathcal{T}\stackrel{G}{\to}\mathcal{T''}$ is said to be a \emph{short exact sequence (up to direct summands)} if $F$ is fully faithful, $G\circ F=0$ and the induced functor $\overline{G}\colon \mathcal{T}/\mathcal{T'}\to\mathcal{T''}$ is an equivalence (up to direct summands). 

\begin{thm}[{\cite[Theorem 2.1]{Nee92}}]
\label{thm:recollement-vs-ses}
%

Given a recollement \eqref{diag:rec} of compactly generated triangulated categories, then both $j_!$ and $i^*$ preserve compact objects and there exists a short exact sequence up to direct summands:
 $${\calT'}^{c} \stackrel{i^*}{\leftarrow}  {\calT}^{c} \stackrel{j_!}{\leftarrow} {\calT''}^{c}.$$
\end{thm}

\begin{remark}
\label{Remark: from small to large}
The converse of the above result is also true for algebraic triangulated categories (=stable categories of Frobenius categories). More precisely, given   a short exact sequence (up to direct summands)
   of small algebraic triangulated categories:
 $$\calP'\stackrel{i^*}{\leftarrow}  \calP \stackrel{j_!}{\leftarrow} \calP''$$
  then there exists   a recollement \eqref{diag:rec} of compactly generated algebraic triangulated categories  such that
$${\calT'}^{c}\simeq\calP', {\calT}^{c}\simeq\calP,  {\calT''}^{c}\simeq\calP'' $$
(these equivalences are up to direct summands).

This result is folklore. Let us give a sketch of proof.  By Keller's dg lifting (see \cite{Kel94}), there exists a small dg category $\calA$ such that its derived category  $\scrD(\calA)$ has its full subcategory of compact objects equivalent to $\calP$ up to direct summands.  Since $\calP''$ is a set of compact objects in  $\scrD(\calA)$, let $\calT''=\Tria(\calP'')$. Then  ${\calT''}^{c}\simeq\calP''$ and  by   \cite[Theorem 2.1]{Nee92}, there exists   a recollement \eqref{diag:rec} such that
${\calT'}^{c}\simeq\calP'$ with $\calT'=\calT/\calT''$.

\end{remark}

Recollements of derived categories of rings are often related to stratifying  idempotents.
Let $R$ be a ring and $e$ be an idempotent element of $R$. Recall that the ideal $ReR$ (or the idempotent $e$) is called \emph{stratifying} (see \cite[2.1.1 Definition ]{CPS96}), if the following two conditions hold:
\begin{itemize}
\item[(SI1)] The multiplication map $Re\otimes_{eRe}eR \to ReR$ is an isomorphism;

\item[(SI2)] ${\rm Tor}^{eRe}_i(Re,eR)=0$, for all $i\ge 1$.
\end{itemize}
The surjective ring homomorphism $\lambda:R\to R/ReR$ induces a fully faithful functor $\lambda_{*}: \Mod R/ReR\to \Mod R$. Then it is known from \cite[Remark 2.1.2]{CPS96} that the ideal $ReR$ is stratifying if and only if $\lambda:R\to R/ReR$ is a homological ring epimorphism if and only if the induced functor $D(\lambda_*):\scrD(R/ReR)\to \scrD(R)$ is fully faithful. In this case, there exists a recollement:
\begin{align} \label{diag:rec2}
 \xymatrixcolsep{4pc}\xymatrix{\scrD(R/ReR) \ar@<0ex>[r]|{D(\lambda_*)}  &\scrD(R) \ar@<-2ex>[l]|{?\otimes^{\mathbb{L}}_{R}Re}
\ar@<2ex> [l]  \ar@<0ex>[r]|{?\otimes^{\mathbb{L}}_{R}R/ReR}  &\scrD(eRe) \ar@<-2ex>[l]|{?\otimes^{\mathbb{L}}_{eRe}eR} \ar@<2ex>[l]
  }
  \end{align}
which  by Theorem~\ref{thm:recollement-vs-ses},   induces  a short exact sequence up to direct summands:
$$\scrK^b(\proj R/ReR) \stackrel{?\otimes^{\mathbb{L}}_{R}R/ReR}{\longleftarrow}  \scrK^b(\proj R ) \stackrel{?\otimes^{\mathbb{L}}_{eRe}eR}{\longleftarrow}  \scrK^b(\proj eRe ).$$

Let us recall the notion of ladders of triangulated categories which is a refinement  of recollements.

\begin{defn} [{\cite[Definition 1.2.1]{BGS88}}]
A \emph{ladder} of $\mathcal{T}$ by $\mathcal{T'}$ and $\mathcal{T''}$ is a  diagram with finite or infinite triangle functors
\vspace{-10pt}$${\setlength{\unitlength}{0.7pt}
\begin{picture}(200,170)
\put(0,70){$\xymatrix@!=3pc{\mathcal{T}'
\ar@<+2.5ex>[r]\ar@<-2.5ex>[r] &\mathcal{T}
\ar[l] \ar@<-5.0ex>[l]\ar@<+5.0ex>[l]
\ar@<+2.5ex>[r] \ar@<-2.5ex>[r] &
\mathcal{T}''\ar[l] \ar@<-5.0ex>[l]
\ar@<+5.0ex>[l]}$}
\put(52.5,10){$\vdots$}
\put(137.5,10){$\vdots$}
\put(52.5,130){$\vdots$}
\put(137.5,130){$\vdots$}
\end{picture}}$$
such that any three consecutive rows form a recollement. 
The \emph{height} of a ladder is the number of recollements
contained in it (counted with multiplicities).
\end{defn}

The hight of a ladder  is an element of $\mathbb{N}\cup \{0,\infty\}$. A recollement is  exactly a ladder of height $1$.

We say that a recollement \eqref{diag:rec} \emph{extends one step downwards}  if both $i^!$ and $j_*$ have right adjoint functors. 
Actually, we have the following notion of extending the recollement one step upwards.

\begin{lem}[{\cite[Lemma 2.11]{JYZ23}}]
\label{lem:extending-recollement-downward}
Given a recollement \eqref{diag:rec}, assume that $\mathcal{T}$, $\mathcal{T}'$ and $\mathcal{T}''$ admit small coproducts and they are compactly generated. Then the following conditions are equivalent:
\begin{itemize}
\item[$(1)$] $i_*$ sends compact objects to compact objects,

\item[$(2)$] $j^*$ sends compact objects to compact objects,

\item[$(3)$] $i^!$ has a right adjoint,

\item[$(4)$] $j_*$ has a right adjoint,

\item[$(5)$] the recollement extends one step downwards.
\end{itemize}
In this case, we can obtain a ladder of $\mathcal{T}$ by $\mathcal{T}'$ and $\mathcal{T}''$ whose height is equal to $2$.
\end{lem}

\subsection{Keller's localisation theorem and splitting theorem for Hochschild homology}\

In  this subsection, we recall Keller's localisation theorem and splitting theorem for Hochschild homology which is one of main ingredients of the proof of Theorem \ref{thm:reduction}.


\begin{thm}[{\cite[3.1 Theorem]{Kel98}\cite[2.4 Theorem (c)]{Kel99}}]\label{hh-der}
Let $\bbA$, $\bbB$ and $\bbC$ be dg algebras which are flat over $\bfk$. Suppose that there is a recollement among the derived categories
$\scrD(\bbA)$, $\scrD(\bbB)$ and $\scrD(\bbC)$ of dg algebras $\bbA$, $\bbB$ and $\bbC$:
\begin{align*}
\xymatrixcolsep{4pc}\xymatrix{
\mathscr{D}(\bbB) \ar@<0ex>[r] &\mathscr{D}(\bbA) \ar@<-2ex>[l]\ar@<2ex>[l] \ar@<0ex>[r]  &\mathscr{D}(\bbC). \ar@<-2ex>[l] \ar@<2ex>[l]
}
\end{align*}
Then there exists a long exact sequence
$$\cdots\rar \HH_n(\bbC)\rar \HH_n(\bbA)\rar \HH_n(\bbB)\rar \HH_{n-1}(\bbC)
\cdots\rar \HH_0(\bbC)\rar \HH_0(\bbA)\rar \HH_0(\bbB)\rar 0.$$
\end{thm}

As an application of the above localisation theorem, B. Keller showed that when the recollement becomes a ladder of height two, the long exact sequence of Hochschild homology groups splits.
\begin{thm}[{\cite[2.9 Proposition (b)]{Kel98}}]\label{thm:splitting-dg}
Let $\bbA$, $\bbB$ and $\bbC$ be dg $k$-algebras  which are flat over $\bfk$ with a ladder of height $2$
$$\xymatrixcolsep{4pc}
\xymatrix{
\scrD(\bbB) \ar@<0.8ex>[r]\ar@<-2.4ex>[r]
&\scrD(\bbA) \ar@<-2.4ex>[r]  \ar@<-2.4ex>[l]
\ar@<0.8ex>[l] \ar@<0.8ex>[r]
&\scrD(\bbC). \ar@<-2.4ex>[l] \ar@<0.8ex>[l]
}$$
Then   for all $n\in \mathbb{N}$,  $$\HH_n(\bbA)\simeq \HH_n(\bbB) \oplus  \HH_n(\bbC).$$
\end{thm}

\begin{remark}  The above theorem could be stated in a more general setup, say the homotopy category of small differential graded categories in the sense of G. Tabuada (see \cite{Tab05a, Tab05b}).
Roughly speaking, a localising homological invariant such as Hochschild homology, cyclic homology, K-theory etc, will have such a splitting property, see \cite[Th\'eor\`eme 6.3(4)]{Tab05b}.
So Han's conjecture could also be stated in this setup (see \cite{Ste24}).
\end{remark}

\section{A reduction theorem for Han's conjecture with applications to  Morita context rings and exact contexts}
\label{sec-han-conj}
In this section, we first prove Theorem \ref{thm:reduction} in the introduction. Then we apply Theorem \ref{thm:reduction} to ladders induced by Morita context rings and exact contexts.

\subsection{A reduction theorem for Han's conjecture}

Recall that  Han's conjecture for a ring $R$: if the Hochschild homology groups $\HH_n(R)$ vanish for all sufficiently large $n$, then $R$ has finite global dimension.

The following observation is a direct consequence of \cite[Theorem I]{AKLY17a} and  Theorem~\ref{hh-der}. For reader's convenience, we include a proof here.

\begin{lem}\label{han-rec-1}
Let $R$, $S$ and $T$ be rings. Assume that there is a recollement of unbounded derived categories
\begin{align*}
\xymatrixcolsep{4pc}
\xymatrix{
\mathscr{D}(S) \ar@<0ex>[r]
&\mathscr{D}(R) \ar@<-2ex>[l] \ar@<2ex>[l] \ar@<0ex>[r]  &\mathscr{D}(T). \ar@<-2ex>[l] \ar@<2ex>[l]
}
\end{align*}
If Han's conjecture holds for $R$, then Han's conjecture holds for $S\times T$.
\end{lem}
\begin{proof}
Suppose that $R$ satisfies Han's conjecture. If the Hochschild homology groups $\HH_n(S)$ and $\HH_n(T)$ vanish for all sufficiently large $n$, then it follows from Theorem~\ref{hh-der}
that the Hochschild homology groups $\HH_n(R)$ vanish for all sufficiently large $n$. Since $R$ satisfies Han's conjecture, $R$ has finite global dimension. By \cite[Theorem I]{AKLY17a}, we get that $S$ and $T$ have finite global dimensions. Thus $S\times T$ satisfies Han's conjecture.
\end{proof}

\begin{proof}[{\bf Proof of Theorem \ref{thm:reduction}}]
By \cite[Theorem I]{AKLY17a}, we get the global dimension of $R$ is finite if and only if those of $S$ and $T$ are finite. Applying Theorem \ref{thm:split}, we obtain that the Hochschild homology groups $\HH_n(R)$ vanish for all sufficiently large $n$ if and only if the Hochschild homology groups $\HH_n(S)$ and $\HH_n(T)$ vanish for all sufficiently large $n$.
Then Theorem \ref{thm:reduction} follows immediately.
\end{proof}

\begin{remark} For three rings $R, S, T$, the existence of a ladder of height two
\[ \xymatrixcolsep{4pc}\xymatrix{\scrD(S) \ar@<1ex>[r]|{i_*=i_!} \ar@<-3ex>[r]  &\scrD(R) \ar@<-3ex>[r]  \ar@<-3ex>[l]|{i^*}
\ar@<1ex> [l]  \ar@<1ex>[r]|{j^!=j^*}  &\scrD(T). \ar@<-3ex>[l]|{j_!} \ar@<1ex>[l]
  } \]
is equivalent to say that there exists a left  recollement (or a colocalisation sequence):
$$\xymatrixcolsep{4pc}
\xymatrix{
\scrK^b(\proj S) \ar@<-1ex>[r]|{i_*=i_!}
&\scrK^b(\proj R)    \ar@<-1ex>[l]|{i^*}
  \ar@<-1ex>[r]|{j^!=j^*}
&\scrK^b(\proj T).  \ar@<-1ex>[l]|{j_!}
}$$
The folklore fact can be shown by using Remark~\ref{Remark: from small to large} and Theorem~\ref{thm:recollement-vs-ses}.
\end{remark}

\subsection{Applications to   Morita context rings}\label{subsec-han-conj-Morita}\

In this subsection, we apply Theorem \ref{thm:reduction} to ladders arising from Morita context rings.

Let $B$ and $C$ be rings, $_BN_C$ and $_CM_B$ be $B$-$C$-bimodule and $C$-$B$-bimodule respectively, and $\alpha: N\otimes_C M \to B$ and $\beta: M\otimes_B N \to C$ be $B$-bimodule and $C$-bimodule homomorphisms respectively. Recall that the sextuple $(B,C,N,M,\alpha,\beta)$ is a Morita context (see \cite{Morita58}) if $$\alpha(n\otimes m) n' = n \beta(m\otimes n') \mbox{ and } \beta (m\otimes n) m' = m\alpha (n\otimes m')\mbox{ for all }m,m'\in M\mbox{ and }n,n'\in N.$$
Associated with a Morita context, we define the \emph{Morita context ring} $A:=\left(\begin{smallmatrix}                                            B & _BN_C \\                                              _CM_B & C                                            \end{smallmatrix}\right)_{(\alpha, \beta)}$ (see \cite{Morita58}), where the addition of elements of $A$ is componentwise and the multiplication is the matrix multiplication, namely
$$\begin{pmatrix}
  b & n \\
  m & c
\end{pmatrix}\begin{pmatrix}
  b' & n' \\
  m' & c'
\end{pmatrix}=\begin{pmatrix}
 bb'+\alpha(n\otimes m') & bn'+nc' \\
 mb'+cm' & \beta(m\otimes n')+ cc'
 \end{pmatrix}.$$
Let $e=\left(\begin{smallmatrix}
1_B & 0 \\
0 & 0
\end{smallmatrix}\right)$ and $f=\left(\begin{smallmatrix}
0 & 0 \\
0 & 1_C
\end{smallmatrix}\right)$ be the idempotent elements of $A$. Then we obtain easily that $eAe= B$, $A/AeA= C/{\rm Im}(\beta)$, $fAf=C$ and $A/AfA= B/{\rm Im}(\alpha)$, where ${\rm Im}(\alpha)$ and ${\rm Im}(\beta)$ denote the image of $\alpha$ and $\beta$, respectively.

\begin{thm}\label{Hanconj-Morita context}
Let $A$ be the above Morita context ring.

$(1)$ If $\beta:M\otimes_BN\to C$ is monomorphism, ${\rm Tor}^B_i(M,N)=0$ for all $i\ge 1$, and $M_B\in \scrD^c(B)$, then Han's conjecture holds for $A$ if and only if it holds for $B\times C/{\rm Im}(\beta)$.

$(2)$ If $\alpha:N\otimes_CM\to B$ is monomorphism, ${\rm Tor}^C_i(N,M)=0$ for all $i\ge 1$, and $N_C\in \scrD^c(C)$, then Han's conjecture holds for $A$ if and only if it holds for $B/{\rm Im}(\alpha)\times C$.
\end{thm}
\begin{proof}
We only prove (1) since a similar argument works for (2). An easy computation shows that $Ae = eAe\oplus fAe\simeq B\oplus M$ as right $B$-modules and $eA = eAe\oplus eAf\simeq B\oplus N$ as left $B$-modules. Due to ${{\rm Tor}^B_i(M,N)}=0$ for all $i\ge 1$, we obtain
$${\Tor}^{eAe}_i (Ae, eA) \simeq {\Tor}^{B}_i (B\oplus M, B\oplus N) \simeq {\Tor}^{B}_i (M, N)=0.$$
Note that the canonical map $\mu_{A}: Ae\otimes_{eAe}eA\rar A$ is a monomorphism if and only if the map $fAe\otimes_{eAe}eA\rar fA$ is a monomorphism if and only if the map $fAe\otimes_{eAe}eAf\rar fAf$ is a monomorphism, i.e. the map $\beta: M\otimes_BN\rar C$ is a monomorphism. By assumption, the ideal $AeA$ is stratifying and there exists a recollement
\begin{equation}\label{rec:Morita context}
\xymatrixcolsep{4pc}\xymatrix{
\scrD(C/{\rm Im}(\beta)) \ar@<0ex>[r]
&\scrD(A) \ar@<-2ex>[l]\ar@<2ex>[l] \ar@<0ex>[r]|{?\otimes^{\mathbb{L}}_{A}Ae}
&\scrD(B). \ar@<-2ex>[l]|{?\otimes^{\mathbb{L}}_{eAe}eA} \ar@<2ex>[l]
}
\end{equation}
As $M_B\in \scrD^c(B)$, we get $A\otimes^{\mathbb{L}}_{A}Ae_B\simeq Ae_B\simeq (B\oplus M)_B$, which is compact in { $\scrD(B)$}. This implies that $?\otimes^{\mathbb{L}}_{A}Ae$ sends compact objects to compact objects.
By Lemma \ref{lem:extending-recollement-downward}, the recollement (\ref{rec:Morita context}) extends one step downwards and there is a ladder of $\scrD(A)$ by $\scrD(C/{\rm Im}(\beta))$ and $\scrD(B)$ whose height is $2$. Now, the statement follows from Theorem \ref{thm:reduction}.
\end{proof}

By { \cite[Example 3.4]{AKLY17a}}, a triangular matrix ring $A:=\left(\begin{smallmatrix}                                            B &  0 \\                                              _CM_B & C                                            \end{smallmatrix}\right)$ induces a ladder of height $2$:
$$\xymatrixcolsep{4pc}
\xymatrix{
\scrD(B) \ar@<0.8ex>[r]\ar@<-2.4ex>[r]
&\scrD(A) \ar@<-2.4ex>[r]  \ar@<-2.4ex>[l]
\ar@<0.8ex>[l] \ar@<0.8ex>[r]
&\scrD(C). \ar@<-2.4ex>[l] \ar@<0.8ex>[l]
}$$

As a direct consequence of Theorem \ref{Hanconj-Morita context}, we have the following result.
\begin{cor}[{\cite[Theorem 2.21]{CRS21}}]\label{han-tri}
Han's conjecture holds for a triangular matrix ring if and only if it holds for the diagonal subring.
\end{cor}

\subsection{Applications to  exact contexts}\label{subsec-han-conj-exactcontext}\

In this section, we apply Theorem \ref{thm:reduction} to
ladders  arising from exact contexts introduced in \cite{CX19}. This kind of ladders involves noncommutative localizations in ring theory, which occur often in algebraic topology and representation theory (see \cite{Ran06, Sch85}).

Let $(\lambda,\mu,M,m)$ be a fixed exact context, where $\lambda:R\to S$ and $\mu:R\to T$ are ring homomorphisms, and  $M$ is an $S$-$T$-bimodule with an element $m\in M$. Let $T\boxtimes_RS$ be the noncommutative tensor
product of $(\lambda,\mu,M,m)$. Note that $T\boxtimes_RS$ has $T\otimes_RS$ as its abelian group, while its multiplication is different from the usual tensor product (see { \cite[Section 3]{CX19}} for details). Define
$$\Lambda:=\begin{pmatrix}
S& M\\ 0 & T\end{pmatrix},\;
\Gamma:=\begin{pmatrix}
T\boxtimes_RS& T\boxtimes_RS\\
T\boxtimes_RS& T\boxtimes_RS
\end{pmatrix},\;
\theta:=\begin{pmatrix}
\rho & \beta\\
0 & \phi
\end{pmatrix}: \; \Lambda \longrightarrow \Gamma,$$
where $\rho: S\to T\boxtimes_RS,\; s\mapsto 1\otimes s$ for $s\in S$, $\phi: T\to T\boxtimes_RS,\; t\mapsto t\otimes 1$ for $t\in
T$, and $\beta: M\rightarrow T\otimes_RS$ is the unique $R$-$R$-bimodule homomorphism such that $\phi=\beta(m\cdot)$ and $\rho=\beta(\cdot m)$.

Let $\varphi:(0,T)\longrightarrow (S,M),\;(0,t)\mapsto (0, mt) \mbox{ for } t\in T$. Then $\varphi$ is a homomorphism of $R$-$\Lambda$-bimodules. By $\cpx{P}$ we denote the mapping cone of $\varphi$. Clearly, $\cpx{P}$ is a complex in $\mathscr{C}(R^{\op} \otimes_{\mathbb{Z}}\Lambda)$ and $\cpx{P}$ is perfect in { $\mathscr{D}(\Lambda)$; see \cite[Section 5.2]{CX19} for example}.

{ An exact context $(\lambda,\mu,M,m)$ is said to be \emph{homological} if ${\rm Tor}^R_i(T,S)=0$ for all $i\geq 1$; see \cite[Section 1]{CX21} for example. Recall from \cite[Lemma 3.2 (2)]{CX21} that $(\lambda,\mu,M,m)$ is homological if and only if  $\theta:\Lambda\to \Gamma$ is a homological ring epimorphism.}
\begin{thm}
\label{Hanconj-exact context}
Let $(\lambda, \mu, M, m)$ be a homological exact context. Assume $T_R\in\scrD^c(R)$, that is, $T_R$ has a finite projective resolution by finitely generated projective $R$-modules.
Then Han's conjecture holds for $(T\boxtimes_RS) \times R$ if and only if it holds for $S\times T$.
\end{thm}

\begin{proof}
{ By \cite[Theorem 1.1]{CX21}, there is a recollement of derived categories of rings
	\begin{align}\label{rec:exact context1}
		\xymatrixcolsep{4pc}\xymatrix{
			\mathscr{D}(T\boxtimes_RS) \ar@<0ex>[r]|{} &\mathscr{D}(\Lambda) \ar@<-2ex>[l]_{} \ar@<2ex>[l]^{} \ar@<0ex>[r]|{j^!=j^*}  &\mathscr{D}(R), \ar@<-2ex>[l]|{j_!} \ar@<2ex>[l]|{}
		}
	\end{align}}
	where $j_!:=?\otimes^{\mathbb{L}}_R\cpx{P}$ and $j^!:=\cpx{\Hom}_{\Lambda}(\cpx{P},?)$. By an easy calculation, we know that
	$$j^{!}(\Lambda)=\cpx{\Hom}_{\Lambda}(\cpx{P},\Lambda)\simeq (T\oplus {\rm Con}(\cdot m))[-1]\in \scrD(R),$$ where ${\rm Con}(\cdot m)$ stands for the two-term complex $0\to S \stackrel{\cdot m}{\rightarrow} M\to 0$ with $S$ of degree $-1$.
	Since the sequence $0\rightarrow R \stackrel{\left({\lambda}\atop{\mu}\right)}{\longrightarrow}
	S\oplus T
	\stackrel{(\cdot m,-m \cdot)}{\longrightarrow}
	M\rightarrow 0$ is exact, we have ${\rm Con}(\cdot m)\simeq {\rm Con}(\mu)$ in $\scrD(R)$, where ${\rm Con}(\mu)$ stands for the two-term complex $0\to R \stackrel{\mu}{\rightarrow} T\to 0$ with $R$ of degree $-1$. Then $$j^{!}(\Lambda)\simeq (T\oplus {\rm Con}(\mu))[-1]\in \scrD(R),$$ Due to $T_R\in \scrD^c(R)$, $j^{!}(\Lambda)$ is compact in $\scrD(R)$. This implies that $j^{!}$ sends compact objects to compact objects.
	By Lemma \ref{lem:extending-recollement-downward}, the recollement (\ref{rec:exact context1}) extends one step downwards and there is a ladder of $\mathscr{D}(\Lambda)$ by $\mathscr{D}(T\boxtimes_RS)$ and $\mathscr{D}(R)$ whose height is 2. Then, Theorem \ref{Hanconj-exact context} follows from Theorem \ref{thm:reduction} and Corollary \ref{han-tri}.
\end{proof}

Now, let us state several consequences of Theorem \ref{Hanconj-exact context}. Firstly, we utilize Theorem \ref{Hanconj-exact context} to Han's conjecture of ring extensions.

\begin{cor}\label{Hanconj-ring extension}
Suppose that $R\subseteq S$ is an extension of rings, that is, $R$ is a subring of the ring $S$ with the same identity. Let $S'$ be the endomorphism ring of the right $R$-module $S/R$. If the right $R$-module $S'$ is projective and finitely generated, then
Han's conjecture holds for $(S'\boxtimes_RS)\times R$ if and only if it holds for $S\times S'$, where $S'\boxtimes_RS$ is the noncommutative tensor product of an exact context defined by the extension.
\end{cor}
\begin{proof}
Let $\lambda:R\rightarrow S$ be the inclusion, $\pi:S\rightarrow S/R$ the
canonical surjection and $\lambda':R\rightarrow S'$ the induced map by right multiplication.
Recall from \cite[Section 4]{CX21} that the quadruple $\big(\lambda, \lambda', \Hom_R(S,S/R),\pi \big)$ is an exact context and the noncommutative tensor product $S'\boxtimes_RS$ of this exact context is well defined. Assume that $S'_R$ is finitely generated and projective. It is clear that  ${\rm Tor}^R_i(S', S)=0$ for all $i\geq 1$.
Then, Corollary \ref{Hanconj-ring extension} follows from Theorem \ref{Hanconj-exact context}.
\end{proof}

Next, we apply Theorem \ref{Hanconj-exact context} to trivial extensions. Recall that, given a ring $R$ and an $R$-$R$-bimodule $M$, the \emph{trivial extension} of $R$ by $M$ is a ring, denoted by $R\ltimes M$, with abelian group $R\oplus M$ and multiplication: $(r,m)(r',m')=(rr',rm'+mr')$ for $r,r'\in R$ and $m,m'\in M$.

\begin{cor}\label{Hanconj-trival extension}
Let $\lambda: R\to S$ be a ring epimorphism and $M$ be an $S$-$S$-bimodule such that ${\rm Tor}^R_i(M,S)$ = $0$ for all $i\geq 1$. Assume $M_R\in \scrD^c(R)$, then
Han's conjecture holds for $(S\ltimes M)\times R$ if and only if it holds for $S\times (R\ltimes M)$.
\end{cor}

\begin{proof}
Let $T:=R\ltimes M$, $\mu: R\rightarrow T$ be
the inclusion from $R$ into $T$. By the proof of \cite[Corollary 3.19]{CX17}, $(\lambda, \mu, S\ltimes M, 1 )$ is a homological exact context and $T\boxtimes_RS\simeq S\ltimes M$ as rings. As $M_R\in \scrD^c(R)$, we get $T_R\in \scrD^c(R)$. Then Corollary \ref{Hanconj-trival extension} follows immediately from Theorem \ref{Hanconj-exact context}.
\end{proof}

Now, we apply Theorem \ref{Hanconj-exact context} to pullback squares of rings and surjective homomorphisms.

\begin{cor}\label{Hanconj-pullback square}
Let $R$ be a ring, and $I_1$, $I_2$ be ideals of $R$ such that $I_1\cap I_2=0$. Assume $R/I_2\in \scrD^c(R)$, then
Han's conjecture holds for $R/(I_1+I_2)  \times R$ if and only if it holds for $R/I_1\times R/I_2$.
\end{cor}
\begin{proof}
We define $\lambda:R\to R/I_1$ and $\mu:R\to R/I_2$ to be the canonical surjective ring homomorphisms. By the proof of \cite[ Corollary 3.20]{CX17}, we know that $(\lambda,\mu,R/(I_1+I_2),1)$ is a homological exact context and $T\boxtimes_RS\simeq R/(I_1+I_2)$ as rings.. Since $R/I_2\in \scrD^c(R)$, then Corollary \ref{Hanconj-pullback square} follows immediately from Theorem \ref{Hanconj-exact context}.
\end{proof}

\section{Han's conjecture for skew-gentle algebras}\label{sec-skew-gentle}

In this section, we will apply our theorem to skew-gentle algebras  and show that Han's conjecture holds for skew-gentle algebras.

Recall a quiver with relations $(Q,I)$ is called a \emph{gentle pair} if the following hold:
\begin{itemize}
\item[(GP1)] Each vertex of $Q$ is start point of at most two arrows, and end point of at most two arrows.
\item[(GP2)] For each arrow $\alpha$ in $Q$, there is at most one arrow $\beta$ with $t(\alpha)=s(\beta)$ such that $\alpha \beta \not\in I$, and at most one arrow $\gamma$ with $t(\gamma)=s(\alpha)$ such that $\gamma\alpha \notin I$.
\item[(GP3)] For each arrow $\alpha$ in $Q$, there is at most one arrow $\beta$ with $t(\alpha)=s(\beta)$ such that $\alpha \beta \in I$, and at most one arrow $\gamma$ with $t(\gamma)=s(\alpha)$ such that $\gamma\alpha \in I$.
\item[(GP4)] The algebra $A(Q,I)$ is finite dimensional.
\end{itemize}
A finite dimensional algebra $A(Q,I)$ is called a \emph{gentle algebra} if $(Q,I)$ is a gentle pair. Let $(Q,I)$ be a gentle pair. We add some special loops in $Q$, and denote the set of special loops by $S_p$. We call $(Q,I,S_p)$ the \emph{skew-gentle triple} if $(Q',I\cup\{\alpha^2|\alpha \in S_p\})$ is  a gentle pair where $Q'$ is the quiver obtained by adding the special loops to $Q$. The finite dimensional algebra $A(Q,I,S_p):=kQ'/\langle I\cup \{\alpha^2-\alpha|\alpha \in S_p\}\rangle$ is called a \emph{skew-gentle algebra} corresponding to the skew-gentle triple $(Q,I,S_p)$.

  \begin{thm}   Han's conjecture holds for skew-gentle algebras.

  \end{thm}

  \begin{proof}   Let $A=A(Q,I,S_p)$ be a skew-gentle algebra with $S_p\neq \emptyset$.   Denote $B=A(Q,I)$.
  By \cite[Theorem 1.1(a)]{Che22}, there is a recollement of derived categories:
  \begin{align*}
\xymatrixcolsep{4pc}\xymatrix{
\mathscr{D}(B) \ar@<0ex>[r] &\mathscr{D}(A) \ar@<-2ex>[l] \ar@<2ex>[l]\ar@<0ex>[r]  &\mathscr{D}(C). \ar@<-2ex>[l]\ar@<2ex>[l]
}
\end{align*}
By \cite[Theorem 1.1(c)]{Che22}, $C$ is a finite dimensional algebra with $\gldim(C) \leq 1$. Hence, this recollement extends one step downwards (in fact extends to an infinite ladder).

Since $\gldim(C) \leq 1$, $\HH_i(C)=0$ for any $i\geq 1$ by  \cite[Proposition 6]{Kel98},  Han's conjecture holds for $C$; since a gentle algebra is a monomial algebra, by \cite[Theorem 3]{Han06}, Han's conjecture holds for $B$. So by Theorem~\ref{thm:reduction}, Han's conjecture holds for $A$.
\end{proof}

\section{Han's conjecture for category algebras of finite EI categories and GLS algebras}\label{sec-category-algebra}

In this section, we will consider two classes of algebras which can be viewed as triangular matrix algebras: category algebras of finite EI categories and GLS algebras associated with Cartan matrices, and show that Han's conjecture holds for these two classes of algebras.

\subsection{Category algebras}\

Let $\bfk$ be a field and $\mathscr{C}$ be a finite category. Here, finite means that $\mathscr{C}$ has
only finitely many morphisms. Denote by ${\rm Mor}\mathscr{C}$ the finite set of all morphisms in $\mathscr{C}$. The \emph{category algebra} $\bfk\mathscr{C}$ of $\mathscr{C}$ is defined as
follows: $\bfk\mathscr{C}=\bigoplus\limits_{\alpha \in {\rm Mor}\mathscr{C}}\bfk\alpha$ as
a $\bfk$-vector space and the product $*$ is given by the rule
\[\alpha * \beta=\left\{\begin{array}{ll}
\alpha\circ\beta, & \text{ if }\text{$\alpha$ and $\beta$ can be composed in $\mathscr{C}$}; \\
0, & \text{otherwise.}
\end{array}\right.\]
The unit is given by $1_{\bfk\mathscr{C}}=\sum\limits_{ x \in {\rm Obj}{\mathscr{C}} }{\rm Id} _{x}$,
where ${\rm Id}_{x}$ is the identity endomorphism of an object $ x$ in $\mathscr{C}$.

The category $\mathscr{C}$ is called a \emph{finite EI category} provided
that all endomorphisms in $\mathscr{C}$ are isomorphisms. In particular,
${{\rm Hom}_{\mathscr{C}}}(x,x)={{\rm Aut}_{\mathscr{C}}}(x)$ is a finite group for any object $x$ in $\mathscr{C}$.

Let $\mathscr{C}$ and $\mathscr D$ be two equivalent finite categories. Then
$\bfk\mathscr{C}$ and $\bfk\mathscr D$ are Morita equivalent; see
\cite[Proposition 2.2]{Webb07}. In particular, $\bfk\mathscr{C}$ is Morita
equivalent to $\bfk\mathscr{C}_0$, where $\mathscr{C}_0$ is any skeleton of $\mathscr{C}$. So we
may assume that $\mathscr{C}$ is skeletal, that is, any two
distinct objects $x$ and $y$ in $\mathscr{C}$ are not isomorphic.

Let $\mathscr{C}$ be a skeletal finite EI category with $n$ ($n\geq 2$) objects. We assume that
${\rm Obj}\mathscr{C}=\{x_1,x_2,\cdots,x_n\}$ satisfying
${\rm Hom}_{\mathscr{C}}(x_i,x_j)=\emptyset$ if $i>j$. Let $M_{ij}:=\bfk{\rm Hom}_{\mathscr{C}}(x_j,x_i)$.
Write $R_i:=M_{ii}$. We observe that
$R_i=({\rm Id}_{x_i})\bfk\mathscr{C}({\rm Id}_{x_i})=\bfk{\rm Aut}_{\mathscr{C}}(x_i)$ is a group algebra.
Then $M_{ij}$ is naturally an $R_i$-$R_j$-bimodule, and we have a
morphism of $R_i$-$R_j$-bimodules $\psi_{ilj}: M_{il}\otimes_{R_l}
M_{lj} \rightarrow M_{ij}$ which is induced by the composition of
morphisms in $\mathscr{C}$.

The category algebra $\bfk\mathscr{C}$ is isomorphic to the corresponding
	triangular matrix algebra
	\begin{align}\label{cat-tri}
	\bfk\mathscr{C}\simeq\Gamma_{\mathscr{C}}=\left(
		\begin{array}{cccc}
		R_1 & &  &  \\
			M_{21}& R_2 & & \\
	\vdots	& 	\vdots & \ddots &  \\
	 M_{n1}&  M_{n2}& \cdots & R_n \\
		\end{array}
		\right)
		\end{align}
whose multiplication is induced by those morphisms $\psi_{ilj}$; compare \cite[Notation 4.1]{Wang16}.

\begin{thm}
	Han's conjecture holds for category algebras of finite EI categories.
\end{thm}

\begin{proof}
Let $\mathscr{C}$ be a finite EI category. Recall from \cite[Theorem 2.2]{Kel96} that Han's conjecture is an invariant under Morita equivalence, so it suffice to assume that $\mathscr{C}$ is skeletal. Assume that $\mathscr{C}$ has $n$ objects. Let $\Gamma_{\mathscr{C}}$ be the corresponding triangular matrix algebra of $\mathscr{C}$. Then by induction on $n$ and by Corollary \ref{han-tri}, Han's conjecture holds for the category algebra $\bfk\mathscr{C}$ if and only if it holds for $R_1 \times R_2 \times \cdots \times R_n$. Recall from \cite[Section 4.1]{Cru23} that Han's conjecture holds for each group algebra of a finite group, which implies that Han's conjecture holds for each $R_i$. Therefore, Han's conjecture holds for $\bfk\mathscr{C}$.
\end{proof}

\subsection{GLS algebras}\

In this subsection, we recall the GLS algebras introduced by Geiss-Leclerc-shr\"{o}er in \cite{GLS17}, which are associated to Cartan triples and show that Han's conjecture holds for GLS algebras.

Let $(C,D,\Omega)$ be a Cartan triple. Here, $C=(c_{ij})\in M_n(\mathbb{Z})$ is a symmetrizable generalized Cartan matrix which provides that the following conditions hold:
\begin{enumerate}
	\item[(C1)] $c_{ii}=2$ for all $i$;
	\item[(C2)] $c_{ij}\leq 0$ for all $i\neq j$, and $c_{ij}=0$ if and only if $c_{ji}=0$;
	\item[(C3)] There is a diagonal  matrix $D={\rm diag}(d_1,\cdots,d_n)\in M_n(\mathbb{Z})$ with $d_i\geq 1$ for all $i$ such that the product matrix $DC$ is symmetric.
\end{enumerate}
The matrix $D$ appearing in (C3) is called a \emph{symmetrizer} of $C$. And $\Omega$ is an orientation of $C$ which is a subset of $ \{1,2,\cdots,n\}\times \{1,2,\cdots,n\}$ such that the following conditions hold:
\begin{enumerate}
	\item[$(1)$] $\{(i,j),(j,i)\}\cap \Omega\neq \emptyset$ if and only if $c_{ij}<0$;
	\item[$(2)$] For each sequence ($i_1,i_2,\cdots,i_t,i_{t+1}$) with $t\geq 1$ and $(i_s,i_{s+1})\in \Omega$ for all $1\leq s\leq t$,  we have $i_1\neq i_{t+1}$.
\end{enumerate}

For a Cartan triple $(C,D,\Omega)$, define a quiver $Q=Q(C,\Omega)$ as follows: the set of vertices $Q_0:=\{1,2, \cdots,n\}$ and the set of arrows
\[Q_1:=\{\alpha^{(g)}_{ij}:j\rightarrow i\mid(i,j)\in \Omega, 1\leq g\leq {\rm gcd}(c_{ij},c_{ji})\}\cup\{\varepsilon_i:i\rightarrow i\mid 1\leq i\leq n\}.\]
Here, ${\rm gcd}(c_{ij},c_{ji})$ means the greatest common divisor of $c_{ij}$ and $c_{ji}$, which is always assumed to be positive.
We call $Q$ a quiver of type $C$.

\begin{defn}{\rm (\cite[Section 1.4]{GLS17})}\label{defn:GLS}
	Let $\bfk$ be a field, and $(C, D, \Omega)$ be a Cartan triple.  Then we have the quiver $Q=Q(C,\Omega)$. Let
	\[H=H(C,D,\Omega)=kQ/I,\]
	where $\bfk Q$ is the path algebra of $Q$, and $I$ is the two-sided ideal of $\bfk Q$ defined by the following relations:
	\begin{enumerate}
		\item[(H1)] For each vertex $i$, we have the \emph{nilpotency relation}
		\[\varepsilon_i^{d_i}=0.\]
		\item[(H2)] For each $(i,j)\in \Omega$ and each $1\leq g\leq {\rm gcd}(c_{ij},c_{ji})$, we have the \emph{commutativity relation}
		\[\varepsilon_i^{\frac{d_i}{{\rm gcd}(d_i,d_j)}} \alpha^{(g)}_{ij}=\alpha^{(g)}_{ij}\varepsilon_j^{\frac{d_j}{{\rm gcd}(d_i,d_j)}}.\]
	\end{enumerate}
\end{defn}
This $H=H(C, D, \Omega)$ is called the \emph{GLS algebra} associated to a Cartan triple $(C, D, \Omega)$.

Fix the Cartan triple $(C, D, \Omega)$  as above and $Q=Q(C,\Omega)$ is the quiver of type $C$. For each $1\leq i,j\leq n$, set $$H_{ij}:={\rm Span}_\bfk\{p\mid p \text{ is a path in } H \text{ from } j \text{ to }i\}.$$
We observe that $H_i:=H_{ii}\simeq \bfk [\varepsilon_i]/(\varepsilon_i^{d_i})$ and $H_{ij}$ is naturally an $H_i$-$H_j$-bimodule, and we have a
morphism of $H_i$-$H_j$-bimodules $\phi_{ilj}: H_{il}\otimes_{H_l}
H_{lj} \rightarrow H_{ij}$ which is induced by the concatenation of
paths in $H$.

For viewing the GLS algebra as a triangular matrix algebra, we renumbering the vertices set $Q_0=\{1,2,\ldots,n\}:=\{x_1,x_2,\ldots,x_n\}$ such that $(x_i,x_j)\notin \Omega$ for $i<j$. That is, there is no arrow from $x_j$ to $x_i$ in $Q$ for $i<j$. Indeed, there is no path from $x_j$ to $x_i$ in $Q$ for $i<j$. Denote by $A_i:=H_{x_i}$ and $A_{ij}:=H_{x_ix_j}$. We observe that the GLS algebra is isomorphic to the corresponding triangular matrix algebra
\begin{align}\label{cat-tri2}
H\simeq\Gamma_{H}=\left(
\begin{array}{cccc}
A_1 & &  &  \\
A_{21}& A_2 & & \\
\vdots	& 	\vdots & \ddots &  \\
A_{n1}&  A_{n2}& \cdots & A_n \\
\end{array}
\right)
\end{align}
whose multiplication is induced by those morphisms $\phi_{ilj}$.

\begin{thm}
	Han's conjecture holds for the GLS algebras associated to Cartan triples.
\end{thm}

\begin{proof}
	Let $H=H(C, D, \Omega)$ be the GLS algebra associated to a Cartan triple $(C, D, \Omega)$ and $\Gamma_H$ be the corresponding triangular matrix algebra of $H$. Recall from \cite[Theorem 3]{Han06} that Han's conjecture  holds for each truncated monomial algebra, which implies that Han's conjecture holds for each $A_i$ in $\Gamma_H$. Therefore, Han's conjecture holds for $H$ by induction on $n$ and by Corollary \ref{han-tri}.
\end{proof}

\bigskip

\noindent
{\bf Acknowledgments.} The second and the fourth authors were supported by the National Natural Science Foundation of China (No. 12071137), by Key Laboratory of Mathematics and Engineering Applications of Ministry of Education, by  Shanghai Key Laboratory of PMMP (No. 22DZ2229014), and by Fundamental Research Funds for the Central Universities. The third author was supported by the National Natural Science Foundation of China (No. 12401038).

We are very grateful to Hong-Xing Chen, Kostia Iusenko, Bernhard Keller, John W.MacQuarrie, Greg Stevenson for useful remarks and comments.

\noindent
{\bf Declaration of interests. } The authors have no conflicts of interest to disclose.

\noindent
{\bf Data availability. } No new data were created or analyzed in this study.

\vspc

\end{document}